\documentclass[preprint,11pt]{elsarticle}

\usepackage{amsfonts}
\usepackage{amsmath}
\usepackage{amsthm}
\usepackage{multirow}
\usepackage{bm}

\usepackage{graphicx}
\usepackage{caption,subcaption}
\usepackage{ctable}

\DeclareMathOperator{\ddiv}{div}
\newtheorem{thm}{Theorem}

\newtheorem{remark}{Remark}

\providecommand{\norm}[1]{\lVert#1\rVert}

\newcommand{\uu}{\bm u}
\newcommand{\qq}{\bm q}

\begin{document}

\begin{frontmatter}

\title{A parallel-in-time fixed-stress splitting method for Biot's consolidation model}

\author[Bergen_address]{Manuel Borregales\corref{mycorrespondingauthor}}\ead{Manuel.Borregales@uib.no}	
\author[Bergen_address]{Kundan Kumar}\ead{Kundan.Kumar@uib.no}
\author[Bergen_address]{Florin Adrian Radu}\ead{Florin.Radu@uib.no}
\author[UZ_address]{Carmen Rodrigo}\ead{carmenr@unizar.es}
\author[CWI_address]{Francisco Jos\'e Gaspar}\ead{F.J.Gaspar@cwi.nl}

\address[Bergen_address]{Department of Mathematics, University of Bergen, All\'egaten 41, 50520 Bergen, Norway}
\address[CWI_address]{CWI, Centrum Wiskunde \& Informatica, Science Park 123, P.O. Box 94079, 1090 Amsterdam, The Netherlands}
\address[UZ_address]{IUMA and Department of Applied Mathematics, University of Zaragoza, Mar\'ia de Luna, 3, 50018 Zaragoza, Spain}

\cortext[mycorrespondingauthor]{Corresponding author}

\begin{abstract}
In this work, we study the parallel-in-time iterative solution of coupled flow and geomechanics in porous media, modelled by a two-field formulation of the Biot's equations. In particular, we propose a new version of the fixed stress splitting method, which has been widely used as solution method of these problems. This new approach forgets about the sequential nature of the temporal variable and considers the time direction as a further direction for parallelization. We present a rigorous convergence analysis of the method and a numerical experiment to demonstrate the robust behaviour of the algorithm.
\end{abstract}


\end{frontmatter}


\section{Introduction}

The coupled poroelastic equations describe the behaviour of fluid-saturated porous materials
undergoing deformation. 
Such coupling has been intensively investigated, starting from the pioneering one-dimensional work of Terzaghi \cite{terzaghi}, which was extended to a more general three-dimensional theory by Biot \cite{Biot1, Biot2}. Biot's model was originally developed to study geophysical applications such as reservoir geomechanics, however, nowadays it is widely used in the modeling of many applications in a great variety of fields, ranging
from geomechanics and petroleum engineering, to biomechanics or food processing.
There is a vast literature on Biot's equations and their existence, uniqueness, and regularity, see Showalter \cite{showalter}, Phillips and Wheeler \cite{phillips1} and the references therein.

Reliable numerical methods for solving poroelastic problems are needed for the accurate
solution of multi-physics phenomena appearing in different application areas. In particular, the solution of the large linear systems of equations arising from the discretization of Biot's model is the most consuming part when real simulations are performed. For this reason, a lot of effort has been made in the last years to design efficient solution methods for these problems. 
Two different approaches can be adopted, the so-called monolithic or fully coupled methods and the iterative
coupling methods. The monolithic approach consists of solving the linear system simultaneously for
all the unknowns. The challenge here, is the design of efficient preconditioners to accelerate the convergence of Krylov subspace methods and the design of efficient smoothers in a multigrid framework. Recent advances in both 
directions can be found in \cite{Bergamaschi2007, Ferronato2010, Gaspar2004, Luo2017} and the references therein. These methods usually provide unconditional stability and convergence.
Iterative coupling methods, however, solve sequentially  the equations
for fluid flow and geomechanics, at each time step, until a converged solution within a prescribed 
tolerance is achieved. They offer several attractive features as their flexibility, for example, since they allow to link two different codes for fluid flow and geomechanics for solving the coupled poroelastic problems. The design of iterative schemes however is an important consideration for an efficient, convergent, and robust algorithm.   
The most used iterative coupling methods are the drained and undrained splits, which solve the mechanical problem first, and the fixed-strain and fixed-stress splits, which on the contrary solve the flow problem first \cite{Kim2009, Kim_PhD}.

Among iterative coupling schemes, the fixed stress splitting method is the most widely used.
This sequential-implicit method basically consists in solving the flow problem first fixing the volumetric mean total stress, and then the mechanics  part is solved from the values obtained at the previous flow step. In the last years, a lot of research has been done on this method. The unconditional stability of the fixed-stress split method is shown in \cite{Kim2011} using a von Neumann analysis. In addition, stability and convergence of the fixed-stress split method have been rigorously established in \cite{Mikelic}. Recently, in \cite{Both2017} the authors have proven 
the convergence of the fixed-stress split method in energy norm for heterogeneous problems.
Estimates for the case of the multirate iterative coupling scheme are obtained in \cite{Almani}, where multiple finer time steps for flow are taken within one coarse mechanics time step, exploiting the different time scales for the mechanics and flow problems. In \cite{Bause}, the convergence of this method is demonstrated in the fully discrete case when space-time finite element methods are used. 
In \cite{Castelleto2015}, the authors present a very interesting approach which consists to re-interpret the fixed-stress split scheme as a preconditioned-Richardson iteration with a particular block-triangular preconditioning operator. Recently, in \cite{Gaspar_Rodrigo2017} an inexact version of the fixed-stress split scheme has been successfully proposed as smoother in a geometric multigrid framework, which provides an efficient monolithic solver for Biot's problem. 

All the  previously mentioned algorithms are based on a time-marching approach, in which each time step is solved after the other in a sequential manner, and therefore they do not allow the parallelization of the
temporal variable. Time parallel time integration methods, however, are receiving a lot of interest nowadays because of the advent of massively parallel systems with thousands of cores, permitting to reduce drastically the computing time \cite{Gander2015}. Due to the mixed elliptic-parabolic structure of Biot's problem, the development of parallel-in time algorithms is not intuitive. In the present work, we introduce a very simple version of the fixed-stress splitting method for the poroelasticity problem which allows an easy parallelization in time.  We further analyze the method and show that it is convergent. Techniques similar with the ones from \cite{Mikelic, Both2017, Bause} are used. For completeness, in Section 3, we include a new proof for the convergence of the fixed-stress split algorithm in the semidiscrete case. The theoretical results are sustained by numerical computations.

The remainder of the paper is organized as follows. In Section \ref{sec:model} we briefly introduce the poroelasticity model and present the considered finite element discretizations. Section \ref{sec:fixed_stress} is devoted to the description of the classical fixed-stress split algorithm. In Section \ref{sec:parallel_fixed}, the parallel-in-time new approach based on the fixed-stress split algorithm is presented and its convergence analysis is derived. 
Section \ref{sec:numerical} illustrates the robustness of the proposed parallel-in-time fixed-stress split method through a numerical experiment. Finally, some conclusions are drawn in Section \ref{sec:conclusion}.

\section{Mathematical model and discretization}\label{sec:model}

The equations describing poroelastic flow and deformation are derived from the principles of
fluid mass conservation and the balance of forces on the porous matrix.
More concretely, according to Biot's theory~\cite{Biot1,Biot2}, and assuming $\Omega$ a bounded open subset of $ {\mathbb R}^d,\; d \leq 3$, with regular
boundary $\Gamma$, the consolidation process must satisfy the following system of partial differential equations:
\begin{eqnarray}
\mbox{\rm equilibrium equation:} & & -{\rm div} \, {\boldsymbol \sigma}' + \alpha \nabla \, p = \rho {\bm g}, \quad {\rm in} \, \Omega, \label{eq11} \\
\mbox{\rm constitutive equation:} & & \bm{\sigma}' = 2 G {\boldsymbol \varepsilon}(\bm{u})  + \lambda\ddiv(\bm{u}) {\bm I}, \quad {\rm in} \, \Omega,
\label{eq12} \\
\mbox{\rm compatibility condition:} & & {\boldsymbol \varepsilon}({\bm u}) = \frac{1}{2}(\nabla {\bm u} + \nabla
{\bm u}^t), \quad {\rm in} \, \Omega,
\label{eq13} \\
\mbox{\rm Darcy's law:} & & {\bm q} = - \frac{1}{\mu_f} {\bm K}
\left(\nabla p  - \rho_f {\bm g} \right), \quad {\rm in} \, \Omega,
\label{eq14} \\
\mbox{\rm continuity equation:} & &
\frac{\partial}{\partial t} \left( \frac{1}{\beta} p + \alpha \nabla \cdot \, \bm u \right) + \nabla \cdot
\, {\bm q} = f, \quad {\rm in} \, \Omega,
\label{eq15}
\end{eqnarray}
where 
$\boldsymbol I$
is the identity tensor, ${\bm u}$ is the displacement vector, $p$ is the pore pressure,
${\boldsymbol \sigma}'$ and ${\boldsymbol \varepsilon}$ are the effective stress and strain tensors
for the porous medium, ${\bm g}$ is the gravity vector, ${\bm q}$ is the percolation velocity of the fluid relative to the soil, $\mu_f$ is the fluid viscosity and  $\boldsymbol K$  is the absolute permeability
tensor. The Lam\'{e} coefficients, $\lambda$ and $G$, can be also expressed in terms of the
Young's modulus $E$ and the Poisson's ratio $\nu$  as $\lambda = E \nu / ((1-2 \nu )(1+\nu))$ 
and  $G= E/(2+2\nu)$. The bulk density  $\rho$ is related to the densities of the solid ($\rho_s$) and fluid ($\rho_f$) phases as $\rho = \phi \rho_f + (1-\phi) \rho_s$, where $\phi$ is the porosity. $\beta$ is the Biot modulus
and $\alpha$ is the Biot coefficient given by $\alpha = 1 - K_b/K_s,$ where $K_b$ is the drained bulk modulus , and $K_s$ is the bulk modulus of the solid phase.

If considering the displacements of the solid matrix ${\bm u}$ and the pressure of the fluid $p$ as primary variables, we obtain the so-called two-field formulation of the Biot's consolidation model. With this idea in mind, combining equations \eqref{eq11}-\eqref{eq15}, the mathematical model can be written as
\begin{eqnarray}
&& -\ddiv\bm{\sigma}' + \alpha \nabla p = \rho {\bm g},\qquad
\bm{\sigma}' = 2 G \ {\boldsymbol \varepsilon}(\bm{u})  + \lambda\ddiv(\bm{u}) \bm{I}, \label{two-field1}\\
&& \frac{\partial}{\partial t} \left( \frac{1}{\beta} p + \alpha \nabla \cdot \, \bm u \right) - \nabla \cdot
\, \left( \frac{1}{\mu_f} {\bm K}
\left(\nabla p  - \rho_f {\bm g} \right)\right) = f.\label{two-field2}
\end{eqnarray}
The most important feature of this mathematical model is that the equations are strongly coupled. Here, the Biot parameter $\alpha$ plays the role of coupling parameter between these equations. 
In order to ensure the existence and uniqueness of solution, we must supplement the system with appropriate boundary and  initial conditions. For instance,
\begin{equation}\label{bound-cond}
\begin{array}{ccccc}
  p = 0, & &  \quad \boldsymbol \sigma' \, {\bm n} &=& {\bm 0}, \quad \hbox {on }\Gamma _t, \\
{\bm u} = {\bm 0}, & & \quad  \displaystyle  {\bm K}  \left(\nabla p  - \rho_f {\bm g} \right) \cdot {\bm n} &=& 0, \quad \hbox {on } \Gamma _c,
\end{array}
\end{equation}
where ${\bm n}$ is the unit outward normal to the boundary and
$\Gamma_t \cup \Gamma_c = \Gamma$, with $\Gamma_t$ and $\Gamma_c$
disjoint subsets of $\Gamma$ having non null measure.
For the initial time, $t=0$, the following condition is fulfilled
\begin{equation}\label{ini-cond}
     \left (\frac{1}{\beta} p + \alpha \nabla \cdot {\bm u} \right)\, (\bm{x},0)=0, \,  \bm{x} \in\Omega.
\end{equation}

\subsection{Semi-discretization in space}
To introduce the spatial discretization of the Biot model, we choose the finite element method. We define the standard Sobolev spaces ${\bm V} = \{{\bm u}\in (H^1(\Omega))^n \ |  \ {\bm u}|_{\Gamma _c} = {\bm 0} \},$
and $Q = \{p \in H^1(\Omega) \ |  \  p|_{\Gamma _t} = 0\}$, with $H^1(\Omega)$ denoting
the Hilbert subspace of $L_2(\Omega)$ of functions with first weak derivatives in
$L_2(\Omega)$. Then, we introduce the variational formulation for the two-field formulation of
the Biot's model as follows: For each $t \in (0,T]$, find $({\bm u}(t) ,p(t)) \in {\bm V} \times Q$ such that
\begin{eqnarray}
  && a(\bm{u}(t),\bm{v}) -  \alpha(p(t),\ddiv \bm{v}) = (\rho \bm g,\bm{v}),
     \quad \forall \  \bm{v}\in \bm V, \label{variational1}\\
  &&  \alpha(\ddiv \partial_t{\bm{u}(t)},q) + \frac{1}{\beta}(\partial_t p(t),q) + b(p(t),q)   = (f,q) +
   ( {\bm K}\mu_f^{-1} \rho_f {\bm g},\nabla q ),
   \quad \forall \ q \in Q,\label{variational2}
\end{eqnarray}
where $(\cdot,\cdot)$ is the standard inner product in the
space $L_2(\Omega)$, and the bilinear forms $a(\cdot,\cdot)$ and $b(\cdot,\cdot)$ are given as
\begin{eqnarray*}\label{bilinear}
a(\bm{u},\bm{v}) &=& 2 G \int_{\Omega}{\boldsymbol \varepsilon}(\bm{u}):{\boldsymbol \varepsilon}(\bm{v}) \, {\rm d} \Omega +
\lambda\int_{\Omega} \ddiv\bm{u}\ddiv\bm{v} \, {\rm d} \Omega, \\
b(p,q) &=&  \int_{\Omega} \frac{{\bm K}}{\mu_f} \nabla p \cdot \nabla q \, {\rm d} \Omega.
\end{eqnarray*}
Finally, the initial condition is given by
\begin{equation}
\left (\frac{1}{\beta} p(0) + \alpha \nabla \cdot {\bm u}(0),q \right) = 0, \quad \forall \ q \in L_2(\Omega).
\end{equation}
It is important to consider a finite element pair of spaces ${\bm V}_h \times Q_h$ satisfying an inf-sup condition. One very simple choice would be the stabilized P1-P1 scheme firstly introduced in \cite{Aguilar2008} and widely analyzed in \cite{RGHZ2016}, in which ${\bm V}_h $ consists of the space of piece-wise (with respect to a triangulation ${\cal T}_h$) linear continuous vector valued functions on $\Omega$ and the space $Q_h$ consists of piece-wise linear continuous scalar valued functions. Other choices would be P2-P1, that is, piece-wise quadratic continuous vector valued functions for displacements and piece-wise linear continuous scalar valued functions for pressure, widely studied by Murad and Loula \cite{MuradLoula92, MuradLoula94, MuradLoulaThome}; or the so-called MINI element \cite{RGHZ2016} in which the only difference is that ${\bm V}_h  = {\bm V}_l \oplus {\bm V}_b$, where ${\bm V}_l$ the space of piece-wise linear continuous vector valued functions and ${\bm V}_b$ is the space of bubble functions. Discrete inf-sup stability conditions and convergence results for the stabilized P1-P1 and the MINI element were recently derived in \cite{RGHZ2016}.

The semi-discretized problem can be written as follows: For each $t\in (0,T]$,
find $({\bm u}_h(t) ,p_h(t)) \in {\bm V}_h \times Q_h$ such that
\begin{eqnarray}
  && a(\bm{u}_h(t),\bm{v}_h) -  \alpha(p_h(t),\ddiv \bm{v}_h) = (\rho \bm g,\bm{v}_h),
     \quad \forall \  \bm{v}_h \in \bm V_h, \label{semi_discrete_variational1}\\
  &&  \alpha(\ddiv \partial_t{\bm{u}_h(t)},q_h) + \frac{1}{\beta} (\partial_t{p_h(t)},q_h) + b(p_h(t),q_h)  = (f_h,q_h) + ( {\bm K}\mu_f^{-1} \rho_f {\bm g},\nabla q_h ),
   \quad \forall \ q_h \in Q_h,\label{semi_discrete_variational2}
\end{eqnarray}
giving rise to the following algebraic/differential equations system,
\begin{equation}\label{system_ODES}
\left[ \begin{array}{cc}
0 & 0 \\
B & M_p
\end{array} \right]
\left[ \begin{array}{c}
\dot{u}_h \\
\dot{p}_h
\end{array} \right]
+
\left[ \begin{array}{cc}
A & B^t \\
0 & - C
\end{array} \right]
\left[ \begin{array}{c}
{u}_h \\
{p}_h
\end{array} \right]
= \left[ \begin{array}{c}
{g}_h \\
{\widetilde{f}}_h
\end{array} \right],
\end{equation}
where we have denoted $\dot{u}_h \equiv \partial_t {\bm{u}_h(t)}$ and $\dot{p}_h \equiv \partial_t{p_h(t)}$.

\begin{remark}
We wish to emphasize that the solver based on the fixed stress split method, which we are going to propose in this work, can be applied to other different discretizations of the problem as mixed finite-elements or finite volume schemes, for example.
\end{remark}


\section{The fixed-stress split algorithm for the semi discretized problem}\label{sec:fixed_stress}
A popular alternative for solving the poroelasticity problem in an iterative manner is to solve first the flow problem  supposing a constant volumetric mean total stress, and once the flow problem is solved, the mechanic problem is then exactly solved. This is the so-called fixed-stress split method. More concretely, the volumetric mean total stress is defined as the mean of the trace of the total stress tensor, that is $\sigma_v = {\rm tr}({\bm \sigma}) /3$, and it is related to the volumetric strain $\varepsilon_v = {\rm tr}(\varepsilon)$ as
$\sigma_v = K_b \varepsilon_v - \alpha p.$
By using this relation, we write the flow equation in terms of the volumetric mean total stress instead of the volumetric strain,
\begin{equation}\label{flow2_equation}
\left(\frac{1}{\beta} + \frac{\alpha^2}{K_b} \right) \frac{\partial p}{\partial t}  +
\frac{\alpha}{K_b} \frac{\partial \sigma_v}{\partial t} - \nabla \cdot
\, \left( \frac{1}{\mu_f} {\bm K}
\left(\nabla p  - \rho_f {\bm g} \right)\right) = f.
\end{equation}
Then, the fixed-stress split scheme is based on solving the flow equation considering known
the volumetric mean total stress,
\begin{equation}\label{flow3_equation}
\left(\frac{1}{\beta} + \frac{\alpha^2}{K_b} \right) \frac{\partial p}{\partial t}   - \nabla \cdot
\, \left( \frac{1}{\mu_f} {\bm K}
\left(\nabla p  - \rho_f {\bm g} \right)\right) = f - \alpha \frac{\partial} {\partial t} (\nabla \cdot \, {\bf u} ) 
+ \frac{\alpha^2}{K_b} \frac{\partial p}{\partial t}.
\end{equation}
Finally, instead of the physical parameter $\frac{\alpha^2}{K_b}$, a general parameter $L$ to fix can be considered, obtaining 
\begin{equation}\label{flow4_equation}
\left(\frac{1}{\beta} + L \right) \frac{\partial p}{\partial t}   - \nabla \cdot
\, \left( \frac{1}{\mu_f} {\bm K}
\left(\nabla p  - \rho_f {\bm g} \right)\right) = f - \alpha \frac{\partial} {\partial t} (\nabla \cdot \, {\bf u} ) 
+  L \frac{\partial p}{\partial t}.
\end{equation}
Then, given an initial guess $({\bm u}_h^0(t),p_h^0(t))$, the  fixed-stress split algorithm gives us a sequence of approximations $({\bm u}_h^i(t),p_h^i(t))$, $i \geq 1$ as follows: \\

\noindent {\bf Step 1:} Given $({\bm u}_h^{i-1}(t),p_h^{i-1}(t)) \in {\bm V}_h \times Q_h$, find $p_h^i(t) \in Q_h$ such that
\begin{eqnarray}  
  &&  (\frac{1}{\beta} +L) (\partial_t{p_h^i(t)},q_h) + b(p_h^i(t),q_h)  + \alpha(\ddiv \partial_t{\bm{u}^{i-1}_h(t)},q_h) = L (\partial_t{p_h^{i-1}(t)},q_h) + \nonumber \\
 & & \qquad \qquad   (f_h,q_h) + ( {\bm K}\mu_f^{-1} \rho_f {\bm g},\nabla q_h ),
   \quad \forall \ q_h \in Q_h.\label{discrete_variational_split_pressure}
\end{eqnarray}
{\bf Step 2:} Given $p_h^{i}(t) \in Q_h$, find ${\bm u}_h^{i}(t) \in {\bm V}_h$ such that
\begin{equation} \label{discrete_variational_split_displacement}
a(\bm{u}_h^i(t),\bm{v}_h) =  \alpha(p_h^i(t),\ddiv \bm{v}_h) + (\rho \bm g,\bm{v}_h),
     \quad \forall \  \bm{v}_h \in \bm V_h. 
\end{equation}     
The algorithm starts with an initial approximation $({\bm u}_h^0(t),p_h^0(t))$ defined along the
whole time-interval. A natural choice is to take this approximation constant and equal
to the values specified by the initial condition, $({\bm u}_h^0(t),p_h^0(t)) = (\bm u_0, p_0), \; t \in (0,T].$
\subsection{Convergence analysis in the semidiscrete case}
Let $\delta \bm{u}_h^i(t) = \bm{u}_h^i(t) - \bm{u}_h^{i-1}(t)$ and
$\delta p_h^i(t) = p_h^i(t)-p_h^{i-1}(t)$ denote the difference between two
succesive approximations for displacements and for pressure, respectively.
\begin{thm}
The fixed-stress split method given in \eqref{discrete_variational_split_pressure}-\eqref{discrete_variational_split_displacement} converges for any $L \ge \frac{\alpha^2}{2(\frac{2G}{d}+\lambda)}$. There holds
\begin{equation}
\label{thm_contraction_semidiscrete}
\int_0^t \| \partial_t \delta p_h^i(s) \|^2 \, {\rm d} s \leq 
\frac{L}{(\frac{1}{\beta}+L)}
\int_0^t \| \partial_t \delta p_h^{i-1}(s) \|^2 \, {\rm d} s.
\end{equation}
\end{thm}
\begin{proof}
We take the time derivative of the difference of two successive iterates of the mechanic equation \eqref{discrete_variational_split_displacement}  and test the resulting equation by 
${\bf v}_h = \partial_t \delta  {\bf u}_h^{i-1}$ to get
\begin{equation} \label{test_mech}
2 G ({\boldsymbol \varepsilon}( \partial_t \delta  {\bf u}_
h^{i}),{\boldsymbol \varepsilon}( \partial_t \delta  {\bf u}_h^{i-1}))
+ \lambda (\nabla \cdot \partial_t \delta  {\bf u}_
h^{i},\nabla \cdot \partial_t \delta  {\bf u}_
h^{i-1}) - \alpha (\partial_t \delta p_h^i,\nabla \cdot \partial_t \delta  {\bf u}_
h^{i-1}) = 0.
\end{equation}
By taking the difference between two successive iterates of the flow equation
 \eqref{discrete_variational_split_pressure} and testing with $q_h = \partial_t \delta p_h^i$, we obtain 
\begin{equation} \label{test_flow}
\frac{1}{\beta} \| \partial_t \delta p_h^i \|^2 + L (\partial_t(\delta p_h^i - \delta p_h^{i-1}), \partial_t \delta p_h^i) +  b( \delta p_h^i,\partial_t \delta p_h^i) +\alpha (\nabla \cdot \partial_t \delta  {\bf u}_
h^{i-1},\partial_t \delta p_h^i) = 0.
\end{equation}
After summing up equations \eqref{test_mech} and \eqref{test_flow}, and using the identities
$$
(\sigma,\xi) = \frac{1}{4} \| \sigma+\xi \|^2-\frac{1}{4} \| \sigma-\xi \|^2, \quad  (\sigma-\xi,\sigma) = \|\sigma\|^2 - \| \xi \|^2+\| \sigma-\xi \|^2,
$$
one has
\begin{eqnarray}
& & \frac{G}{2} \| {\boldsymbol \varepsilon} ( \partial_t \delta  {\bf u}_h^{i} + \partial_t \delta  {\bf u}_h^{i-1}) \|^2 + \frac{\lambda}{4} \| \nabla \cdot ( \partial_t \delta  {\bf u}_h^{i} + \partial_t \delta  {\bf u}_h^{i-1}) \|^2 + \frac{1}{\beta} \| \partial_t \delta p_h^i \|^2 + \frac{1}{2} \frac{{\rm d}}{{\rm d} t} \| \delta p_h^i \|_B^2   \nonumber  \\
& & + \frac{L}{2} (\| \partial_t \delta p_h^i \|^2 - \| \partial_t \delta p_h^{i-1} \|^2  
+ \| \partial_t \delta p_h^i -  \partial_t \delta p_h^{i-1}\|^2) = 
\frac{G}{2} \| {\boldsymbol \varepsilon} ( \partial_t \delta  {\bf u}_h^{i} - \partial_t \delta  {\bf u}_h^{i-1}) \|^2 + \nonumber \\
& & \frac{\lambda}{4} \| \nabla \cdot ( \partial_t \delta  {\bf u}_h^{i} - \partial_t \delta  {\bf u}_h^{i-1}) \|^2. \label{intermediate_formula}
\end{eqnarray}
Next, we consider the time derivative of the difference of two successive iterates of the mechanic equation \eqref{discrete_variational_split_displacement}  and test  by 
${\bf v}_h = \partial_t \delta  {\bf u}_h^{i} -  \partial_t \delta  {\bf u}_h^{i-1}$. By applying the Cauchy-Schwarz inequality, it follows
\begin{equation}
\label{ineq_div_pressure}
 \| \nabla \cdot ( \partial_t \delta  {\bf u}_h^{i} - \partial_t \delta  {\bf u}_h^{i-1}) \| \leq
\frac{\alpha}{\frac{2G}{d}+\lambda} \| \partial_t \delta p_h^i -  \partial_t \delta p_h^{i-1}\|.
\end{equation}
Inserting equality \eqref{test_mech} into equation  \eqref{intermediate_formula} and by applying  Cauchy-Schwarz  and \eqref{ineq_div_pressure} inequalities, we obtain 
 \begin{eqnarray}
& & \frac{G}{2} \| {\boldsymbol \varepsilon} ( \partial_t \delta  {\bf u}_h^{i} + \partial_t \delta  {\bf u}_h^{i-1}) \|^2 + \frac{\lambda}{4} \| \nabla \cdot ( \partial_t \delta  {\bf u}_h^{i} + \partial_t \delta  {\bf u}_h^{i-1}) \|^2 + \frac{1}{\beta} \| \partial_t \delta p_h^i \|^2 + \frac{1}{2} \frac{{\rm d}}{{\rm d} t} \| \delta p_h^i \|_B^2   \nonumber  \\
& &  + \frac{L}{2}(\| \partial_t \delta p_h^i \|^2 
+ \| \partial_t \delta p_h^i -  \partial_t \delta p_h^{i-1}\|^2) \leq \frac{L}{2}  \| \partial_t \delta p_h^{i-1} \|^2   +
\frac{\alpha^2}{4(\frac{2G}{d}+\lambda)} \| \partial_t \delta p_h^i -  \partial_t \delta p_h^{i-1}\|^2. \nonumber
\end{eqnarray}
Discarding the first three positive terms, taking $L \ge  \frac{\alpha^2}{2(\frac{2G}{d}+\lambda)}$, and integrating from $0$ to $t$ we finally obtain \eqref{thm_contraction_semidiscrete}. It implies that the scheme is a contraction and therefore convergent. This completes the proof.
\end{proof}

\begin{remark}
It is easy to see that the fixed-stress split method in the semidiscrete case is an iterative method based on a suitable splitting for solving the differential/algebraic equation system \eqref{system_ODES}. In detail, the iterative method can be written in the form
\begin{equation}
\left[\!
\begin{array}{cc}
0 & 0 \\
0 & (1+L) M_p
\end{array}
\!\right]
\left[\!
\begin{array}{c}
\dot{u}_h^i \\
\dot{p}_h^i
\end{array}
\!\right]
\!+\!
\left[\!
\begin{array}{cc}
A & B^t \\
0 & - C
\end{array}
\!\right]
\left[\!
\begin{array}{c}
{u}_h^i \\
{p}_h^i
\end{array}
\!\right]
\!= \!\left[\!
\begin{array}{cc}
0 & 0 \\
-B & L M_p
\end{array}
\!\right] \left[\!
\begin{array}{c}
\dot{u}_h^{i-1} \\
\dot{p}_h^{i-1}
\end{array}
\!\right]
\!+\!
\left[\!
\begin{array}{c}
{g}_h \\
{\widetilde{f}}_h
\end{array}
\!\right].
\label{splitting_DAE}
\end{equation}
\end{remark}


\section{The parallel in time fixed-stress split algorithm for the fully discretized problem}  \label{sec:parallel_fixed}
\subsection{Parallel-in-time algorithm}
For time discretization we use the backward Euler method on a uniform partition $\left\{t_0, t_1, \ldots, t_N\right\}$ of the time interval $(0,T]$ with constant time-step size $\tau$, $N \tau = T$. Then, we have the following fully discrete scheme corresponding to \eqref{semi_discrete_variational1}-\eqref{semi_discrete_variational2}: For $n = 1, 2, \ldots, N$, find $({\bm u}_h^n ,p_h^n) \in {\bm V}_h \times Q_h$ such that
\begin{eqnarray}
  && a(\bm{u}_h^n,\bm{v}_h) -  \alpha(p_h^n,\ddiv \bm{v}_h) = (\rho \bm g,\bm{v}_h),
     \quad \forall \  \bm{v}_h \in \bm V_h, \label{total_discrete_variational1}\\
  &&  \alpha(\ddiv \bar{\partial}_t{\bm{u}_h^n},q_h) + \frac{1}{\beta} ( \bar{\partial}_t{p_h^n},q_h) + b(p_h^n,q_h)  = (f_h^n,q_h) + ( {\bm K}\mu_f^{-1} \rho_f {\bm g},\nabla q_h ),
   \; \forall \ q_h \in Q_h,\label{total_discrete_variational2}
\end{eqnarray}
where $\bar{\partial}_t{\bm{u}_h^n} := (\bm{u}_h^n - \bm{u}_h^{n-1})/\tau$ and $\bar{\partial}_t{p_h^n} := (p_h^n - p_h^{n-1})/\tau$. \\

We now discuss a parallel-in time version of the fixed-stress split method. This algorithm arises in a natural way from the iterative method \eqref{splitting_DAE} by discretizing in time. In this way, given an initial guess $\{({\bm u}_h^{n,0},p_h^{n,0}), n = 0,1, \ldots, N\}$, the  new fixed-stress split algorithm gives us a sequence of approximations $\{({\bm u}_h^{n,i},p_h^{n,i}), n = 0,1, \ldots,N\}$, $i \geq 1$, as follows: \\

\noindent {\bf Step 1:} Given $\{({\bm u}_h^{n,i-1},p_h^{n,i-1}), n = 0,1, \ldots,N\}$, find $p_h^{n,i} \in Q_h, n = 1, \ldots,N,$ such that
\begin{eqnarray}  
  &&  \left(\frac{1}{\beta} +L\right) \left(\frac{p_h^{n,i}-p_h^{n-1,i}}{\tau},q_h\right) + b(p_h^{n,i},q_h) = \alpha \left(\ddiv \frac{\bm{u}_h^{n,i-1}-\bm{u}_h^{n-1,i-1}}{\tau},q_h\right) + \nonumber \\
 & & \qquad \qquad  L \left(\frac{p_h^{n,i-1}-p_h^{n-1,i-1}}{\tau},q_h\right)+ (f_h^n,q_h) + ( {\bm K}\mu_f^{-1} \rho_f {\bm g},\nabla q_h ),
   \quad \forall \ q_h \in Q_h, \label{total_discrete_variational_split_pressure} \\
& & p_h^{0,i} = p_0. \nonumber   
\end{eqnarray}
{\bf Step 2:} Given $p_h^{n,i} \in Q_h, n = 1, \ldots,N,$ find ${\bm u}_h^{n,i} \in {\bm V}_h, n = 1, \ldots,N,$ such that
\begin{equation} \label{total_discrete_variational_split_displacement}
a(\bm{u}_h^{n,i},\bm{v}_h) =  \alpha(p_h^{n,i}, \ddiv \bm{v}_h) + (\rho \bm g,\bm{v}_h),
     \quad \forall \  \bm{v}_h \in \bm V_h. 
\end{equation}

\begin{remark}
We wish to emphasize that the proposed method is parallel-in-time in contrast to the classical
sequential fixed-stress split method based on time-stepping, given as follows: \\ 

\hrule
\vspace{0.2cm}
\noindent For $n = 1, 2, \ldots, N$\\ [-3.5ex]

\;  For $i=1,2,\ldots$\\ [-3.5ex]

\quad \quad {\bf Step 1:} Given $({\bm u}_h^{n,i-1},p_h^{n,i-1}) \in {\bm V}_h \times Q_h$, find $p_h^{n,i} \in Q_h$ such that
\begin{eqnarray*}  
  &&  \left(\frac{1}{\beta} +L\right) \left(\frac{p_h^{n,i}-p_h^{n-1,i}}{\tau},q_h\right) + b(p_h^{n,i},q_h) = \alpha \left(\ddiv \frac{\bm{u}_h^{n,i-1}-\bm{u}_h^{n-1,i-1}}{\tau},q_h\right) + \nonumber \\
 & & \qquad \qquad  L \left(\frac{p_h^{n,i-1}-p_h^{n-1,i-1}}{\tau},q_h\right)+ (f_h^n,q_h) + ( {\bm K}\mu_f^{-1} \rho_f {\bm g},\nabla q_h ),
   \quad \forall \ q_h \in Q_h, \label{total_discrete_variational_split_pressure_2} 
\end{eqnarray*}

\quad \quad {\bf Step 2:} Given $p_h^{n,i} \in Q_h$, find ${\bm u}_h^{n,i} \in {\bm V}_h$ such that
\begin{equation*} \label{total_discrete_variational_split_displacement_2}
a(\bm{u}_h^{n,i},\bm{v}_h) =  \alpha(p_h^{n,i}, \ddiv \bm{v}_h) + (\rho \bm g,\bm{v}_h),
     \quad \forall \  \bm{v}_h \in \bm V_h. 
\end{equation*}

\; End \\ [-3.5ex]

\noindent End\\
\hrule
\vspace{0.5cm}
Notice that in the Step 2 of the new algorithm, $N-1$ independent elliptic problems can be solved in parallel. 
\end{remark} 

\subsection{Convergence analysis}
Let $\delta \bm{u}_h^{n,i} = \bm{u}_h^{n,i} - \bm{u}_h^{n,i-1}$ and
$\delta p_h^{n,i} = p_h^{n,i}-p_h^{n,i-1}$ denote the difference between two
succesive approximations for displacements and for pressure, respectively.
\begin{thm}
The fixed-stress split method given in \eqref{total_discrete_variational_split_pressure}-\eqref{total_discrete_variational_split_displacement}
is convergent for any stabilization parameter  $L \ge \frac{\alpha^2}{2(\frac{2G}{d}+\lambda)}$. There holds
\begin{equation}
\label{thm_contraction_fullydiscrete}
\sum_{n=1}^N \tau \| \bar{\partial}_t \delta p_h^{n,i} \|^2 \, \leq 
\frac{L}{(\frac{1}{\beta}+L)}
\sum_{n=1}^N \tau \| \bar{\partial}_t \delta p_h^{n,i-1} \|^2.
\end{equation}
\end{thm}
\begin{proof}
We take the difference of two successive iterates of the mechanic equation \eqref{total_discrete_variational_split_displacement}  and test the resulting equation by 
${\bf v}_h = \bar{\partial}_t \delta  {\bf u}_h^{n,i-1}$ to get for $n=1, 2,\ldots, N,$
\begin{equation} \label{test_mech_fully}
2 G ({\boldsymbol \varepsilon}( \bar{\partial}_t  \delta  {\bf u}_
h^{n,i}),{\boldsymbol \varepsilon}( \bar{\partial}_t \delta  {\bf u}_h^{n,i-1}))
+ \lambda (\nabla \cdot \bar{\partial}_t  \delta  {\bf u}_
h^{n,i},\nabla \cdot \bar{\partial}_t  \delta  {\bf u}_
h^{n,i-1}) - \alpha (\bar{\partial}_t  \delta p_h^{n,i},\nabla \cdot \bar{\partial}_t  \delta  {\bf u}_
h^{n,i-1}) = 0.
\end{equation}
By taking the difference between two successive iterates of the flow equation
 \eqref{total_discrete_variational_split_pressure} and testing with $q_h = \bar{\partial}_t  \delta p_h^{n,i}$, we obtain 
\begin{equation} \label{test_flow_fully}
\frac{1}{\beta} \| \bar{\partial}_t  \delta p_h^{n,i} \|^2 + L (\bar{\partial}_t (\delta p_h^{n,i} - \delta p_h^{n,i-1}), \bar{\partial}_t  \delta p_h^{n,i}) +  b( \delta p_h^{n,i},\bar{\partial}_t  \delta p_h^{n,i}) +\alpha (\nabla \cdot \bar{\partial}_t  \delta  {\bf u}_h^{n,i-1},\bar{\partial}_t  \delta p_h^{n,i}) = 0.
\end{equation}
After summing up equations \eqref{test_mech_fully} and \eqref{test_flow_fully}, and using the identities
$$
(\sigma,\xi) = \frac{1}{4} \| \sigma+\xi \|^2-\frac{1}{4} \| \sigma-\xi \|^2, \quad  (\sigma-\xi,\sigma) = \|\sigma\|^2 - \| \xi \|^2+\| \sigma-\xi \|^2,
$$
one has
\begin{eqnarray}
& & \frac{G}{2} \| {\boldsymbol \varepsilon} ( \bar{\partial}_t  \delta  {\bf u}_h^{n,i} + \bar{\partial}_t \delta  {\bf u}_h^{n,i-1}) \|^2 + \frac{\lambda}{4} \| \nabla \cdot ( \bar{\partial}_t  \delta  {\bf u}_h^{n,i} + \bar{\partial}_t \delta  {\bf u}_h^{n,i-1}) \|^2 + \frac{1}{\beta} \| \bar{\partial}_t  \delta p_h^{n,i} \|^2    \nonumber  \\
& & + \frac{L}{2} (\| \bar{\partial}_t  \delta p_h^{n,i} \|^2  
+ \| \bar{\partial}_t  \delta p_h^{n,i} -  \bar{\partial}_t  \delta p_h^{n,i-1}\|^2) 
+ \frac{1}{2\tau} (\| \delta p_h^{n,i} \|_B^2  + \| \delta p_h^{n,i} - \delta p_h^{n-1,i} \|_B^2)  
\nonumber \\
& & = \frac{G}{2} \| {\boldsymbol \varepsilon} ( \bar{\partial}_t  \delta  {\bf u}_h^{n,i} - \bar{\partial}_t  \delta  {\bf u}_h^{n,i-1}) \|^2 + \frac{\lambda}{4} \| \nabla \cdot ( \bar{\partial}_t  \delta  {\bf u}_h^{n,i} - \bar{\partial}_t  \delta  {\bf u}_h^{n,i-1}) \|^2 + \frac{L}{2} \| \bar{\partial}_t  \delta p_h^{n,i-1} \|^2  \nonumber \\
& & + \frac{1}{2\tau}  \| \delta p_h^{n-1,i} \|_B^2. \label{intermediate_formula_fully}
\end{eqnarray}
Next, we consider the difference of two successive iterates of the mechanic equation \eqref{total_discrete_variational_split_displacement}  and test  by 
${\bf v}_h =  \bar{\partial}_t \delta  {\bf u}_h^{n,i} -  \bar{\partial}_t \delta  {\bf u}_h^{n,i-1}$ to get
\begin{equation}\label{equality_2_fully}
2 G \| {\boldsymbol \varepsilon}( \bar{\partial}_t \delta  {\bf u}_
h^{n,i}- \bar{\partial}_t \delta  {\bf u}_
h^{n,i-1}) \|^2 + \lambda \| \nabla \cdot ( \bar{\partial}_t \delta  {\bf u}_
h^{n,i}- \bar{\partial}_t \delta  {\bf u}_
h^{n,i}) \|^2 = \alpha ( \bar{\partial}_t \delta p_h^{n,i}-  \bar{\partial}_t \delta p_h^{n,i-1},\nabla \cdot ( \bar{\partial}_t \delta  {\bf u}_h^{n,i} - \bar{\partial}_t \delta  {\bf u}_h^{n,i-1})).
\end{equation}
From this equality, by applying Cauchy-Schwarz inequality, it is easy to see 
\begin{equation}
\label{ineq_div_pressure_fully}
 \| \nabla \cdot (  \bar{\partial}_t \delta  {\bf u}_h^{n,i} -  \bar{\partial}_t \delta  {\bf u}_h^{n,i-1}) \| \leq
\frac{\alpha}{\frac{2G}{d}+\lambda} \|  \bar{\partial}_t \delta p_h^{n,i} -   \bar{\partial}_t \delta p_h^{n,i-1}\|.
\end{equation}
Inserting equality \eqref{equality_2_fully} into equation  \eqref{intermediate_formula_fully} and by applying  Cauchy-Schwarz  inequality and \eqref{ineq_div_pressure_fully}, we obtain 
 \begin{eqnarray}
& & \frac{G}{2} \| {\boldsymbol \varepsilon} ( \partial_t \delta  {\bf u}_h^{n,i} + \partial_t \delta  {\bf u}_h^{n,i-1}) \|^2 + \frac{\lambda}{4} \| \nabla \cdot ( \partial_t \delta  {\bf u}_h^{n,i} + \partial_t \delta  {\bf u}_h^{n,i-1}) \|^2 + \frac{1}{\beta} \| \partial_t \delta p_h^{n,i} \|^2  \nonumber  \\
& &  + \frac{L}{2}(\| \partial_t \delta p_h^{n,i} \|^2 
+ \| \partial_t \delta p_h^{n,i} -  \partial_t \delta p_h^{n,i-1}\|^2) + \frac{1}{2\tau} (\| \delta p_h^{n,i} \|_B^2  + \| \delta p_h^{n,i} - \delta p_h^{n-1,i} \|_B^2)  \nonumber \\
& & \leq \frac{L}{2}  \| \partial_t \delta p_h^{n,i-1} \|^2   + \frac{1}{2\tau}  \| \delta p_h^{n-1,i} \|_B^2 +
\frac{\alpha^2}{4(\frac{2G}{d}+\lambda)} \| \partial_t \delta p_h^{n,i} -  \partial_t \delta p_h^{n,i-1}\|^2. \nonumber
\end{eqnarray}
Discarding positive terms, taking $L \ge \frac{\alpha^2}{2(\frac{2G}{d}+\lambda)}$, and summing up from $n=1$ to $N$, we finally obtain \eqref{thm_contraction_fullydiscrete}. This implies that the scheme is a contraction and therefore convergent. This completes the proof.
\end{proof}
\begin{remark}
Notice that the values of parameter $L$ result to be the same as in the classical fixed-stress split 
scheme.
\end{remark}


\section{Numerical experiment}\label{sec:numerical}

In this section, we present a numerical experiment with the purpose of illustrating the performance of the 
parallel-in-time fixed stress splitting (PFS) method described in Section \ref{sec:parallel_fixed}. We compare the PFS method with the classical fixed stress splitting (FS), see e.g. \cite{Both2017}.  As test problem, we use Mandel's problem, which is a well-established 2D benchmark problem with a known analytical solution  \cite{abousleiman1996mandel,Mandel}. This problem is very often used in the community for verifying  the implementation and the performance of the numerical schemes, see e.g. \cite{phillips1,Kim2009,RGHZ2016,Wan_2015}.  We implemented the problem in the open-source software package deal.II \cite{DealII}.

Mandel's problem consists in a poroelastic slab of extent $2a$ in the $x$ direction, $2b$ in the $y$ direction, and infinitely long in the z-direction, and is sandwiched between two rigid impermeable plates (see Figure \ref{xMandelProblem}). At time $t=0$, a uniform vertical load of magnitude $2F$ is applied and an equal, but upward force is applied to the bottom plate. This load is supposed to remain constant. The domain is free to drain and stress-free at $x=\pm a$. Gravity is neglected.

For the numerical solution, the symmetry of the problem allows us to use a quarter of the physical domain as computational domain (see Figure \ref{tMandel2Problem2}). Moreover, the rigid plate condition is enforced by adding constrained equations so that vertical displacement $u_y(b,t)$ on the top is equal to a known constant value.
	
\begin{figure}[h!]
    \centering
    \begin{subfigure}{0.5\textwidth}
        \centering
        		\includegraphics[scale=0.6]{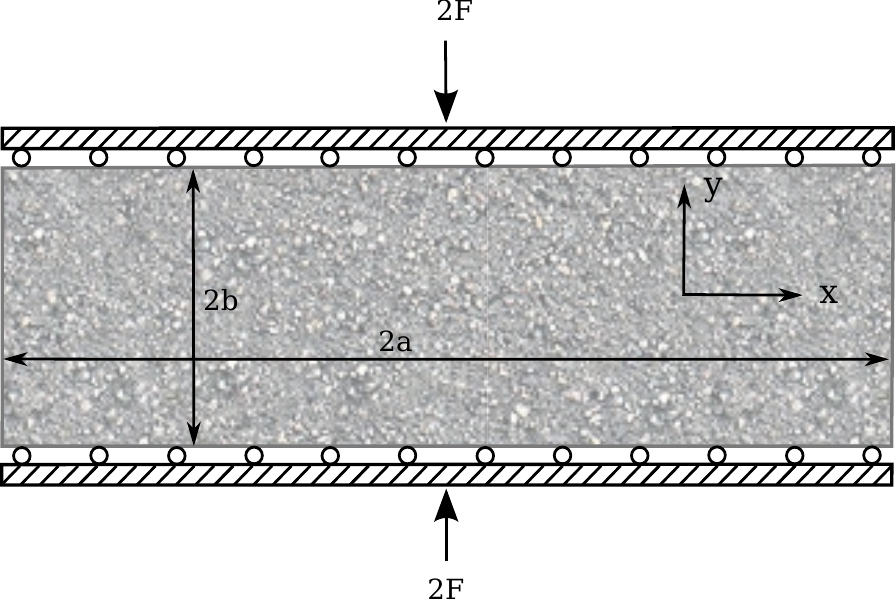} 
        \caption{Mandel's problem physical domain.}
                    \label{xMandelProblem}	
    \end{subfigure}%
    ~ 
    \begin{subfigure}{0.5\textwidth}
        \centering
     		\includegraphics[scale=0.6]{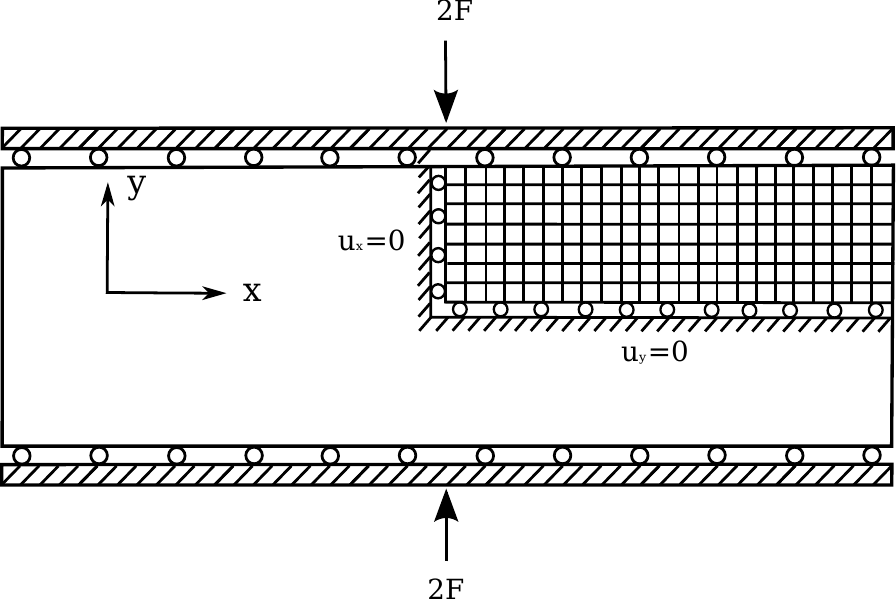} 
        \caption{Mandel's problem quarter domain}
                \label{tMandel2Problem2}
    \end{subfigure}
    \caption{Mandel's problem}
                \label{MandelProblem}	
\end{figure}	

The application of a load (2F) causes an instantaneous and uniform pressure increase throughout the domain \cite{GayX}; this is predicted theoretically 
\cite{abousleiman1996mandel} and it can be used as an initial condition
\begin{align}
p(x,y,0)&=\frac{F B(1+v_u)}{3a}, \nonumber\\
\uu(x,y,0)&=   \begin{pmatrix} \displaystyle \frac{F v_u x}{2G},  &  \displaystyle \frac{-F b(1-v_u)y}{2G a} \end{pmatrix}^\top,\nonumber
\end{align} 
being $B$ the Skempton's coefficient, that for our problem is $B=0.8333$, and 
$\nu_u = \frac{3\nu + B(1-2 \nu)}{3-B(1-2\nu)}$ the undrained Poisson's ratio.

The boundary conditions are specified in Table \ref{BoundaryCondition} and  the input parameters for Mandel's problem are listed in Table \ref{Parameter}.
For all cases, we use as convergence criterion for the schemes
$10^{-8} \norm{\delta p^{n,i}} +10^2\norm{\delta {\uu}^{n,i}} \leq 10^{-8}$, due to the different orders of magnitude among the primary variables.
	 \begin{table}[h!]
\centering
\caption{Boundary conditions for Mandel's problem}
   \resizebox{.55\textwidth}{!}{
  \begin{tabular}{ l  l  l }
    \hline
    Boundary & Flow & Mechanics \\ 
    \specialrule{.1em}{.05em}{.05em} 
    $x=0$ & $ \qq \cdot {\bf n}= 0$ & $\uu \cdot  {\bf n} = 0$ \\
    $y=0$ & $ \qq \cdot {\bf n}= 0$ & $\uu \cdot {\bf n} = 0$ \\
    $x=a$ & $p=0$                  & $  {\bf \sigma} \cdot  {\bf n} = 0$ \\
    $y=b$ & $ \qq \cdot {\bf n} = 0$ & $ {\bf \sigma}_{12}=0$; $\uu \cdot {\bf n}  = u_y(b,t)$  \\
\hline
  \end{tabular}}
  \label{BoundaryCondition}
	\end{table}  

%

	\begin{table}[h!]
\centering
\caption{Input parameter for Mandel's problem}
 \label{Parameter}
  \resizebox{.95\textwidth}{!}{
  \begin{tabular}{ l  l c | llc}
    \hline
    Symbol & Quantity & Value &     Symbol & Quantity & Value \\ 
    \specialrule{.1em}{.05em}{.05em} 
    a & Dimension in $x$ & 100 m&    b & Dimension in $y$ & 10 m  \\
    K & Permeability  & 100 D &   $\mu_f$ & Dynamic viscosity  & 10 $cp$  \\
    $\alpha$ &  Biot's constant  & 1.0 &    $\beta$ & Biot's modulus & $1.65\times 10^{10}$ Pa \\
        $\nu$& Poisson's ratio & $0.4$ &  $E$& Young's modulus & $ 5.94\times 10^9 Pa$  \\
    $B$& Skempton coefficient & $0.83333$ &  $\nu_u$& Undrained Poisson's ratio & $ 0.44$   \\
    $c$ &Diffusivity coefficient& 46.526 $m^2/s$ &
    $F$ & Force intensity & $6.8 \times 10^8 N/m$\\
    $h_x$& Grid spacing  in $x$ & 2.5 m &    $h_y$& Grid spacing  in $y$ & 0.25 m\\ 
    $\tau$& Time step & 1 s &     $T$& Total simulation time  & 32 s\\ 
    \hline
  \end{tabular}
 }
	\end{table}

\begin{figure}[t]
    \centering
    \begin{subfigure}{0.5\textwidth}
        \centering
        \includegraphics[scale=.3]{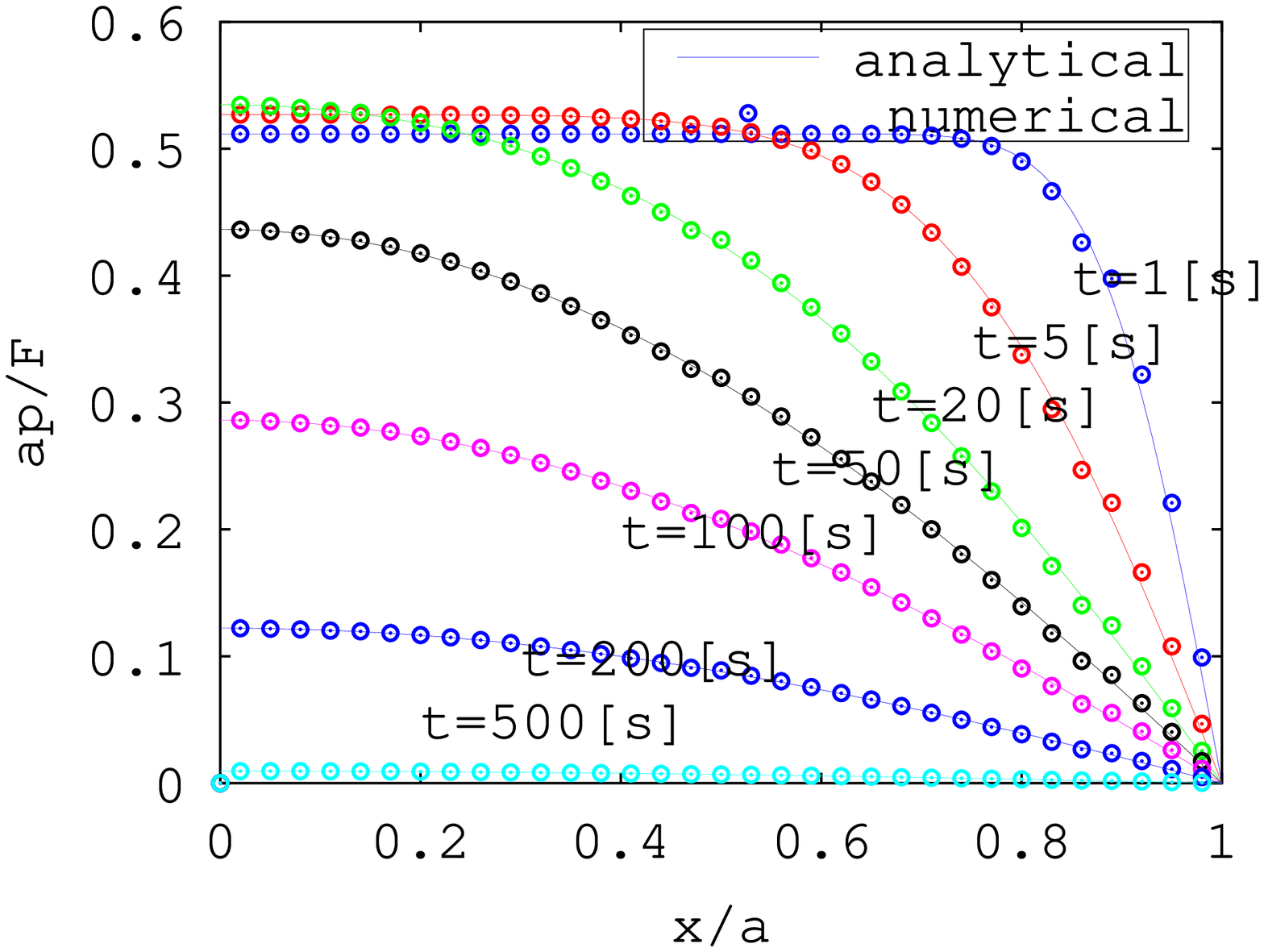} 
        \caption{Pressure solutions}
                \label{solutiona}
    \end{subfigure}%
    ~ 
    \begin{subfigure}{0.5\textwidth}
        \centering
      \includegraphics[scale=0.3]{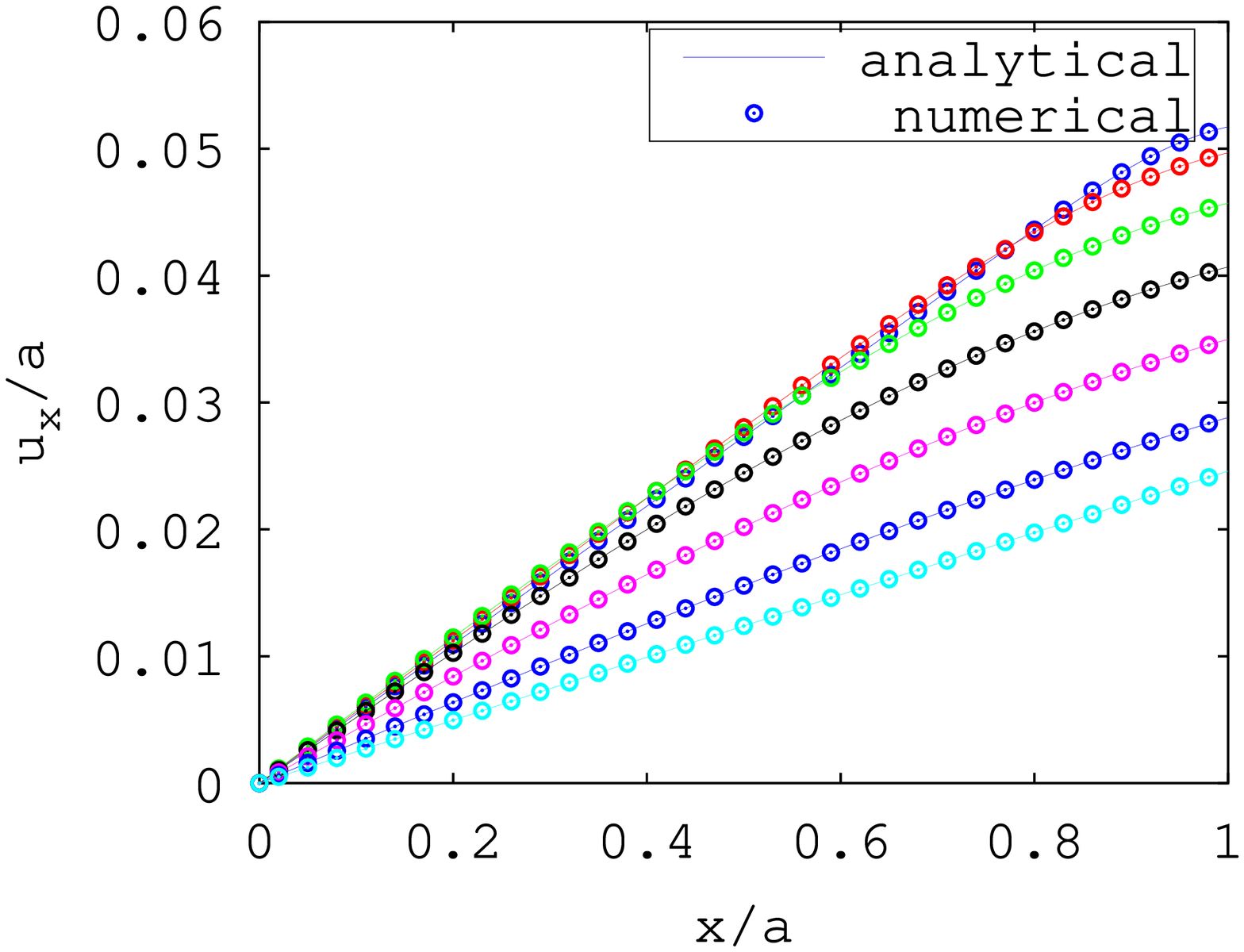}
        \caption{Displacement solutions}
        \label{solutionb}
    \end{subfigure}
    \caption{Comparison of numerical and analytical solutions of the (a) pore pressure and
    (b) displacements for Mandel's problem in different times with $\nu = 0.2$.}
            \label{solution}
\end{figure}

In Figure \ref{solution}, the numerical  and the analytical solutions of Mandel's problem are depicted for different values of time. There is a very good match between both solutions for all cases. Moreover, the results demonstrate the Mandel-Cryer effect, first showing a pressure raise during the first 20 seconds and then, a sudden dissipation throughout the domain.

%
%

The number of iterations for the parallel fixed stress splitting method and  the classical fixed stress splitting method are reported in Figure \ref{Test_Iterations} for different values of parameter $L$ and various values of
$\nu$.  We remark a very similar behaviour of the two methods, with the optimal stabilization parameter $L$ being in this case the physical one $L_{phy} := \alpha^2 / \left(\frac{2G}{d}+\lambda\right)$, see e.g. \cite{Kim2009, Mikelic, Both2017}. 
\begin{figure}[h!]
    \centering
        		\includegraphics[scale=0.35]{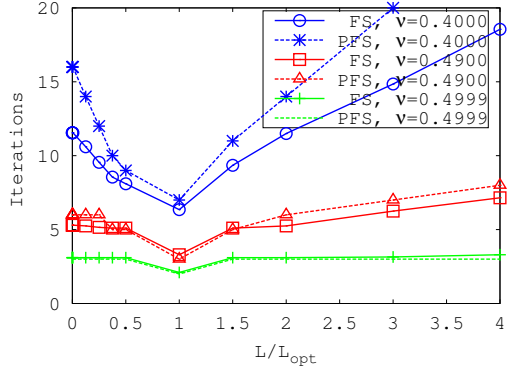} 
        \caption{Performance of the splitting schemes PFS and FS for different values of $L$, $\tau = 1 [s]$, $h = 0.25[m]$. Both schemes has the same optimum value $L_{opt} = L_{phy}$. }
                    \label{Test_Iterations}	
\end{figure}	


Further, in  Figure \ref{Timing} the CPU time reduction for the parallel splitting (PFS) method compared with the classical fixed stress (FS) method is reported in percentage for different values of the Poisson's ratio (see Figure    \ref{Timingcomparison}) and for different time steps (see Figure \ref{Timingcomparison_dt}). The CPU time reduction is estimated by the ratio between the CPU time of PFS and the CPU time of FS. We use parallelism on shared memory machines. In this regard, PFS is set to use one processor to solve the flow problem and up to eight processors in parallel to solve the mechanics, while FS  is set to use one processor to solve flow and mechanics.
For the smallest Poisson's ration ($\nu = 0.4$), PFS is $1\%$ slower than the FS when one processor is used. However, PFS is $10\%$ faster when four processors are used. Still, we notice that the CPU time reduction is limited to the CPU time of the flow problem. In Table \ref{CPU_time_comsumtion} we  report the percentages of CPU which are devoted for solving the flow and the mechanic problems  for different Poisson's ratio values $\nu$. As expected, we clearly observe that when solving the mechanics requires more time, the PFS method becomes more efficient. A reduction of CPU time up to 50\%  (for $\nu \rightarrow 0.5$) is observed (see Figure \ref{Timingcomparison}). Nevertheless, due to the relative small size of the problem, by increasing the number of processors  to more than four we do not see a further significant reduction of the CPU time. This is due to the time lost in the communication between the processors. Finally, we remark that the mesh size and the time step $\tau$ does not influence the number of iterations. This can be seen in Table \ref{DT_iterations}, where we provide the number of iterations for both algorithms, varying the space and time discretization parameters. However,  Figure \ref{Timingcomparison_dt} shows  that the CPU time of PFS slightly improves for smaller time steps $\tau$. This is in agreement with the theory.

\begin{figure}[h!]
    \centering
    \begin{subfigure}{0.5\textwidth}
        \centering
        		\includegraphics[scale=0.3]{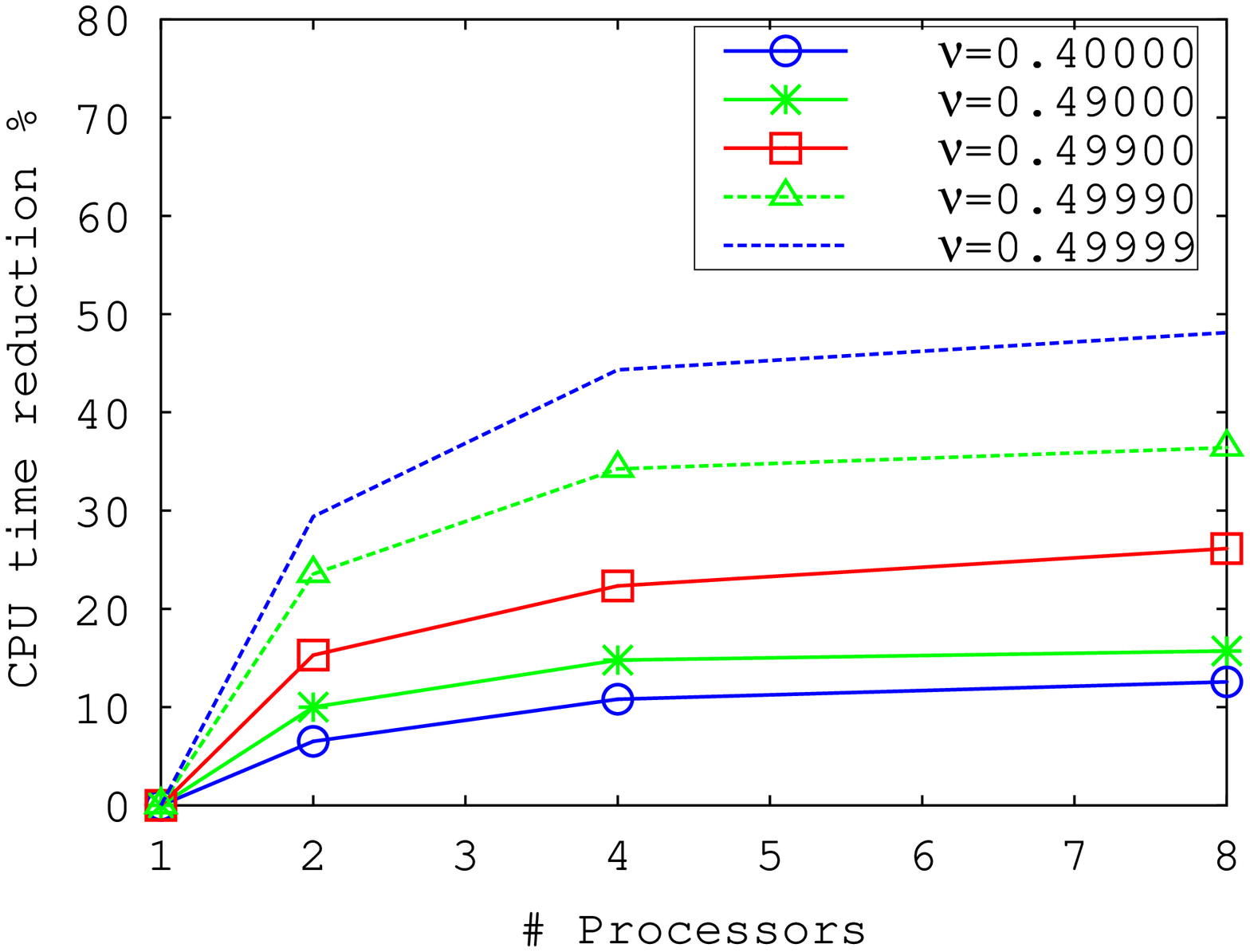} 
        \caption{CPU time reduction for different values of Poisson's ratio $\nu$; $\tau = 1 [s]$, $h = 0.25[m]$. } 
                    \label{Timingcomparison}	
    \end{subfigure}%
    ~ 
    \begin{subfigure}{0.5\textwidth}
        \centering
     		\includegraphics[scale=0.3]{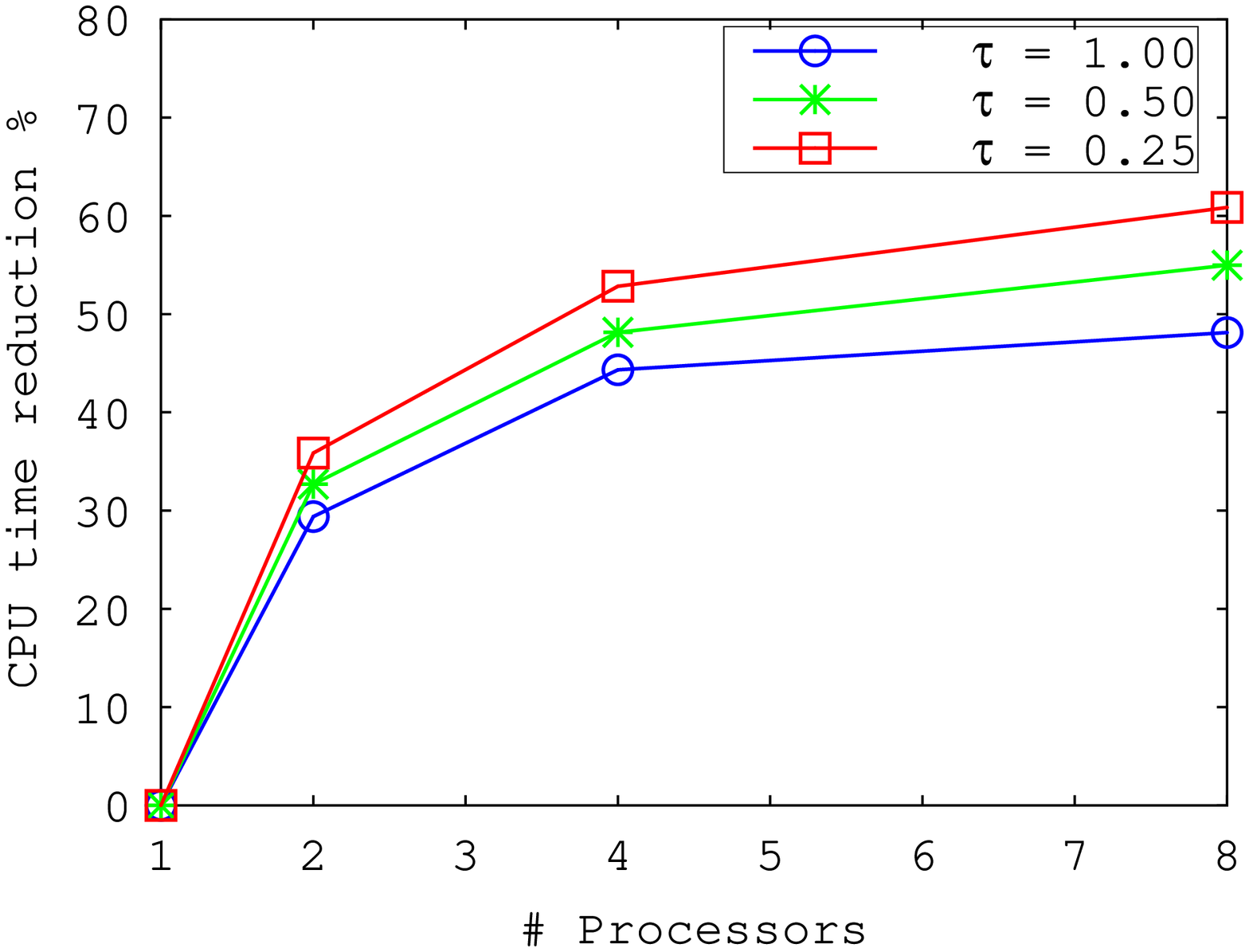} 
        \caption{CPU time reduction for different time step sizes; $\nu =0.49999$, $h = 0.25[m]$. }
           \label{Timingcomparison_dt}	     
    \end{subfigure}
    \caption{Parallel fixed stress splitting CPU times reduction. }
                \label{Timing}	
\end{figure}


%

\begin{table}[htb]
\centering
\caption{CPU time consumption}
\label{CPU_time_comsumtion}
  \resizebox{.55\textwidth}{!}{
  \begin{tabular}{ lcc}
   $\nu$  & Flow problem & Mechanic problem \\
    \hline
0.4 &	$82\%$&$18\%$\\
0.49&	$76\%$&$24\%$\\
0.499&	$68\%$&$32\%$\\
0.4999&	$47\%$&$53\%$\\
0.49999&	$34\%$&$66\%$\\
    \hline
  \end{tabular}
  }
\end{table}

\begin{table}[htb]
\centering
\caption{Number of iterations for different values of $\tau$,$h$, $\nu$}
\label{DT_iterations}
  \resizebox{.65\textwidth}{!}{ 
  \begin{tabular}{ lcc}
  $\nu = 0.49999$\\
    \hline
   $\tau [s]$  & PFS & FS \\
    \hline
1.000 &	2&2.10\\
0.500&	2&2.03\\
0.250&	2&2.02\\
0.125&	2&2.01\\
    \hline
  \end{tabular}  
    \begin{tabular}{ lcc}
    $\nu = 0.499$\\
    \hline
   $h[m]$  & PFS & FS \\
    \hline
    0.5000 &	3&3.20\\
0.2500 &	3&3.20\\
0.1250 &	3&3.19\\
0.0625& 3&3.19\\
    \hline
  \end{tabular}}
\end{table}

\section{Conclusions} \label{sec:conclusion}

We considered the quasi-static Biot's model in the two-field formulation and presented a new fixed stress type splitting method for solving it. The main benefit of the new method is that the mechanics can be solved in a parallel-in-time manner, taking advantage of the current massively parallel computers. We have rigorously analyzed the convergence of the method. If the stabilization term $L$ is chosen big enough, the method is shown to be convergent. The theoretical results are indicating a similar behaviour with the classical fixed stress splitting method (in terms of convergence rate and stabilization parameter). We further performed numerical tests by using the Mandel benchmark problem. The numerical results are confirming a similar behaviour of the parallel fixed stress and the classical fixed stress method in terms of number of iterations and optimal stabilization parameter. With respect to the CPU time, we observe that the new scheme is very efficient (up to 50\% reduction of the CPU time) for problems where solving the mechanics requires more time than solving the flow (in the case of Mandel's problem for Poisson's ratio close to five). Nevertheless, the parallel implementation has still to be optimized. Up to date, only relatively small problems could be solved and therefore the new parallel method remained efficient only for not too many processors (less than eight).





{\bf Acknowledgements} The work of F.~A. Radu and K. Kumar was partially supported by the NFR - Toppforsk project TheMSES \#250223. The work of F.~J. Gaspar is supported by the European Union's
Horizon 2020 research and innovation programme under the Marie
Sklodowska-Curie grant agreement NO 705402, POROSOS. The research of
C.~Rodrigo is supported in part by the Spanish project FEDER /MCYT
MTM2016-75139-R and the DGA (Grupo consolidado PDIE).


\bibliography{mybibfile_ACOMEN}

\end{document}